\title{Geometry of symplectic log Calabi-Yau pairs} 
\documentclass[11pt]{article}

\usepackage{amsfonts}
\usepackage{amsthm}
\usepackage{latexsym, amssymb, bbm}
\usepackage{xypic}
\usepackage{xcolor}
\usepackage{graphicx}
\usepackage{amsmath}
\usepackage{enumerate}
\usepackage[all]{xy}
\usepackage[margin=1.3in]{geometry}

\usepackage{amsmath,amsthm,amsfonts,amssymb,graphicx,psfrag}
\usepackage{color}
\usepackage{hyperref}

\newtheorem{theorem}{Theorem}[section]
\newtheorem{lemma}[theorem]{Lemma}
\newtheorem{cor}[theorem]{Corollary}

\newtheorem{prop}[theorem]{Proposition}
\newtheorem{conj}[theorem]{Conjecture}
\newtheorem{definition}[theorem]{Definition}

\newtheorem{speculation}[theorem]{Speculation}

\newtheorem{example}[theorem]{Example}

\theoremstyle{remark}
\newtheorem{remark}[theorem]{Remark}

\newtheorem*{claim}{Claim}

\numberwithin{equation}{section}

\newcommand{\C}{\mathbb{C}}
\newcommand{\R}{\mathbb{R}}
\newcommand{\Q}{\mathbb{Q}}
\newcommand{\Z}{\mathbb{Z}}

\newcommand{\thmone}{\begin{theorem}}
\newcommand{\thmtwo}{\end{theorem}}
\newcommand{\lemmaone}{\begin{lemma}}
\newcommand{\lemmatwo}{\end{lemma}}
\newcommand{\pfone}{\begin{proof}}
\newcommand{\pftwo}{\end{proof}}
\newcommand{\defone}{\begin{definition}}
\newcommand{\deftwo}{\end{definition}}
\newcommand{\corone}{\begin{cor}}
\newcommand{\cortwo}{\end{cor}}
\newcommand{\cone}{\begin{claim}}
\newcommand{\ctwo}{\end{claim}}
\newcommand{\propone}{\begin{prop}}
\newcommand{\proptwo}{\end{prop}}
\newcommand{\eqone}{\begin{equation}}
\newcommand{\eqtwo}{\end{equation}}
\newcommand{\rmkone}{\begin{remark}}
\newcommand{\rmktwo}{\end{remark}}
\newcommand{\enone}{\begin{enumerate}}
\newcommand{\entwo}{\end{enumerate}}
\newcommand{\itone}{\begin{itemize}}
\newcommand{\ittwo}{\end{itemize}}

\newcommand{\onehalf}{\left(\begin{array}{cc}}
\newcommand{\theother}{\end{array}\right)}

\newcommand{\oneeq}{\begin{equation}}
\newcommand{\twoeq}{\end{equation}}

\usepackage{color}
 \usepackage{comment}

\definecolor{red}{rgb}{1,0,0}          %
\definecolor{green}{rgb}{0,1,0}       %
\definecolor{blue}{rgb}{0,0,1}         
\definecolor{purp}{rgb}{1,0.1,1}       

\author{Tian-Jun Li \& Cheuk Yu Mak}

\AtEndDocument{\bigskip{\footnotesize
  \textsc{School of Mathematics, University of Minnesota, Minneapolis, MN, US} \par  
  \textit{E-mail address}: \texttt{tjli@math.umn.edu} \par
  \addvspace{\medskipamount}
  \textsc{DPMMS, University of Cambridge, Cambridge, UK} \par  
  \textit{E-mail address}: \texttt{cym22@dpmms.cam.ac.uk} \par

}}

\date{May 5, 2018}
 
\begin{document}

\maketitle

\begin{abstract}
 We will survey some aspects of the smooth topology, algebraic geometry, symplectic geometry and contact geometry of anti-canonical pairs in complex dimension two.
\end{abstract}

{\tableofcontents}
\section{Introduction}

Let $Y$ be a smooth rational surface and let $D\subset Y$   be an effective reduced anticanonical divisor.
Such pairs $(Y, D)$, called anti-canonical pairs,  have a rich geometry. 
They were first investigated systematically by Looijenga,  and by Friedman etc  in the 80s. 
Note that $Y-D$ comes with a canonical (up to scaling) nowhere-vanishing 2-form $\Omega$ with simple poles along $D$.
When the intersection matrix of $D$ is negative definite, $D$ can be contracted and $Y$ becomes  a
singular analogue of a K3 surface (a
normal complex analytic surface with  trivial dualizing sheaf). 
   Motivated  by mirror symmetry,  Gross, Hacking and Keel  introduced important new ideas  in a series of papers on log Calabi-Yau varieties, beginning with \cite{GrHaKe11} and \cite{GrHaKe12}.
In particular, they  proved   Torelli type results  in \cite{GrHaKe12} conjectured by Friedman.
In this regard, it was shown in \cite{Pa13} that the symplectic cohomology of $X-D$  is canonically isomorphic to the vector space of global sections of the structure sheaf of its mirror.
Readers are also referred to \cite{Auroux}, \cite{GHKK}, \cite{GHS} and the references therein for more about this mirror symmetry story.

We have a more topological flavour and 
we will survey some other aspects of the smooth topology, algebraic geometry, symplectic geometry and contact geometry of  anti-canonical pairs in Sections 2, 3, 4, 5 respectively. 

Let $X$ be a smooth, oriented 4 dimensional manifold. 
 A topological divisor of $X$ refers to a connected configuration of finitely many closed embedded, oriented, labeled  smooth surfaces $D=C_1 \cup \dots \cup C_k$ 
in    $X$ such that  
each intersection between two surfaces  is transversal and positive,
no three $C_i$  intersect at a common point,
and $D$ has  empty intersection with $ \partial X$.
A topological divisor $D$  is often described by a plumbing graph  with vertices corresponding to the surfaces $C_i$ and
edges corresponding to intersection points.  Associated to $D$ there are plumbed neighborhoods $N_D$ as well as the boundary
plumbed 3-manifold $Y_D$, which are all well-defined 
up to orientation-preserving diffeomorphisms.

Given a topological divisor  $D=C_1 \cup \dots \cup C_k$ 
in    $X$, we      use $[C_i]$ to denote the homology class of $C_i$ in $H_2(X)$ and $H_2(N_D)$,  $r(D)=k$ to denote the length of $D$, and $S(D)=(s_1,\cdots, s_{r(D)})$ to denote the sequence of self-intersection numbers. $H_2(N_D)$ is freely generated by $C_i$. 
The intersection matrix of $D$ is the $k$ by $k$ square matrix $Q_D=(s_{ij}=[C_i]\cdot [C_{j}])$, where $\cdot$ is used for any of the pairings $H_2(X) \times H_2(X), H^2(X) \times H_2(X), H^2(X) \times H^2(X, \partial X)$. Via the Lefschetz duality for  $N_D$, the intersection matrix $Q_D$ can  be identified with the natural  homomorphism $Q_D: H_2(N_D)\to H_2(N_D, Y_D)$.  We use homology and cohomology with $\Z$ coefficient unless otherwise specified. 

For a symplectic 4-manifold $(X, \omega)$ a symplectic divisor is a topological divisor $D$ with each $C_i$  symplectic and having the orientation 
  positive with respect to $\omega$.
 Let $K_{\omega}$ be the symplectic canonical class of $(X, \omega)$. 

\begin{definition}
A {symplectic log Calabi-Yau pair} $(X,D,\omega)$ is a closed symplectic 4-manifold $(X,\omega)$ together with a nonempty symplectic divisor $D=\cup C_i$ representing the  Poincare dual of $-K_{\omega}$.
 A  symplectic log Calabi-Yau pair  is called a A {symplectic Looijenga pair}  if  each $C_i$  is a sphere, called an elliptic log Calabi-Yau pair if $D$  is a torus. 
\end{definition}

Here are some quick observations, which have  well known analogues  in the holomorphic category.  

\begin{lemma}
For a symplectic log Calabi-Yau pair $(X, D, \omega)$, 

$\bullet$ $c_1(X-D, \omega)=0$, and  $(X-D, \omega)$ is minimal in the sense it has no symplectic sphere with self-intersection $-1$. 

$\bullet $  $D=\cup C_i$ is either a torus or a cycle of spheres.

$ \bullet$  $(X, \omega)$ is a rational or elliptic ruled symplectic 4-manifold. In particular,  $\kappa(X, \omega)=-\infty$.
  $D$ is a  cycle of spheres only  when $(X, \omega)$ is rational.
  
  $\bullet$  $b^+(Q_D)=0$ or $1$.

\end{lemma}

\begin{proof} The vanishing of $c_1(X-D)$ follows directly from the definition and 
$X-D$ being minimal follows directly from the adjunction formula. The 2nd bullet is also proved by the adjunction formula. 
Let $g_i$ be the genus of $C_i$. Then 
$$      -[C_i]\cdot [C_i]- \sum_{j\ne i}  [C_j]\cdot [C_i]=K_{\omega}\cdot [C_i]=-[C_i]\cdot [C_i]+2g_i-2.$$
So $2g_i-2=- \sum_{j\ne i}  [C_j]\cdot [C_i] \leq 0$, namely, $g_i\leq 1$ for each $i$. If $g_i=1$ for some $i$, then
$ \sum_{j\ne i}  [C_j]\cdot [C_i]=0$ which implies that $C_i$ is the only component. The remaining case is that $g_i=0$ for each $i$. In this case,  $ \sum_{j\ne i}  [C_j]\cdot [C_i]=2$ for each $i$ and clearly  $D$ is a cycle of spheres. 

Since $D$ is a nonempty symplectic divisor representing $-K_{\omega}$ we have $K_{\omega}\cdot [\omega]< 0$. It  follows from  \cite{Liu96}, \cite{OhOn96}   that $(X, \omega)$ is rational or ruled and admits a genus $0$ Lefschetz fibration over a Riemann surface $\Sigma$. Let $F$ be the fibre class. 
Since $K_{\omega}\cdot F=-2$ and $D$ represents $-K_{\omega}$ the projection of $D$ to $\Sigma$  
 has nonzero degree.  
Since  $D=\cup C_i$ is either a torus or a cycle of spheres, the  genus of $\Sigma$  is at most $1$.

The last bullet follows from the fact that $b^+(X)=1$. 
\end{proof}

Therefore elliptic pairs and  Looijenga pairs  are exactly the symplectic log Calabi-Yau pairs with length $1$ and at least $2$ respectively. 
We remark that symplectic log Calabi-Yau pairs have vanishing  relative symplectic Kodaira dimension (cf. \cite{LiZh11}).
The following is the main result in \cite{LiMa16-deformation}.

\begin{theorem}[Symplectic deformation] \label{thm: symplectic deformation class=homology classes}
Two symplectic log Calabi-Yau pairs 
 are symplectic deformation equivalent  
if they are   homologically equivalent. In particular, each symplectic deformation class  contains a K\"ahler   pair. 

Moreover,  two symplectic log Calabi-Yau pairs are strictly symplectic deformation equivalent if they are strictly homologically equivalent.


\end{theorem}

Let us explain the various equivalence notions in the theorem (See \cite{Sa13} for a thorough discussion of equivalence notions for symplectic manifolds). 
Let $(X^0, D^0, \omega^0)$ and $(X^1, D^1, \omega^1)$ be two  symplectic pairs with $r(D^0)=r(D^1)=k$. 
They  are said to be  homologically equivalent if  
there is an orientation preserving diffeomorphism $\Phi: X^0 \to X^1$ such that $\Phi_*[C^0_j]=[C^1_j]$ for all $j=1,\dots,k$. 
The homological equivalence is said to be strict if, in addition,  $\Phi^*[\omega^1]=[\omega^0]$. 
When $X^0=X^1$,  they are said to be 
symplectic homotopic   if $(D^0, \omega^0)$ and $(D^1, \omega^1)$  are connected by a family of symplectic divisors $(D^t, \omega^t)$, and they are further said to be 
symplectic isotopic if $\omega^t$ can be chosen to be  a constant family. 
$(X^0, D^0, \omega^0)$ and $(X^1, D^1, \omega^1)$  are said to be 
 symplectic deformation equivalent if they are homotopic,  up to an orientation preserving  diffeomorphism. They are said to be 
strictly symplectic deformation equivalent   if they are symplectic isotopic,  up to an orientation preserving  diffeomorphism.  

A sequence $(s_i)$ of integers is said to be  anti-canonical  if it is realized as  $S(D)$
for a symplectic log Calabi-Yau pair $(X, D, \omega)$.  
Combined with Theorem 3.1 in \cite{Fr}, we obtain
\begin{cor} \label{cor: finite deformation} Given a anti-canonical sequence $(s_i)$, there are only finitely many symplectic deformation types of symplectic log Calabi-Yau  pairs $(X, D, \omega)$ with $S(D)=(s_i)$.
\end{cor}


 There is an algorithm to write down the anti-canonical sequences, starting from   the list of minimal pairs  and reverse the minimal reduction process in \cite{LiMa16-deformation}. It is interesting to compare anti-canonical sequences with spherical circular sequences. 
A  spherical circular sequence is  the sequence of  a cycle  of symplectic spheres in a rational surface with minimal   complement. 
An anti-canonical sequence $(s_i)$ is said to be  rigid   if, for any cycle of symplectic spheres $D\subset (X, \omega)$ with $S(D)=(s_i)$  and $(X-D, \omega)$  minimal, $(X, D, \omega)$  is a symplectic log Calabi-Yau pair. 



\begin{theorem} [Anti-canonical sequences, \cite{LiMa17-contact}]
\label{theorem: anti-canonical}
 Each    spherical circular sequence with $b^+=1$     is    anti-canonical, and each  anti-canonical sequence  with $b^+=1$ is  rigid. 
\end{theorem}

From the contact  point of view, symplectic log Calabi-Yau pairs are separated into $3$ groups, as stated in  the following theorem.   Here,  $Kod(Y, \xi)$ is the contact Kodaira dimension introduced in \cite{LiMa16}. 

\begin{theorem}
[Contact trichotomy, \cite{LiMa17-contact}] 
\label{prop: convex-concave}
Let $(X,D,\omega)$ be a symplectic log Calabi-Yau pair, $Q_D$  the intersection matrix of $D$ and $(s_i)$ the self intersection sequence.

(i) If $Q_D$ is negative definite, then $D$ admits {\rm convex} neighborhoods inducing the same contact  $3-$manifold $(Y_D, \xi_D)$, which  only depends on  $S(D)$ and has  $Kod\leq 0$. 

(ii) If $b^+(Q_D)= 1$,  up to local symplectic deformations, $D$ admits {\rm concave} neighborhoods inducing the same  contact  $3-$manifold $(Y_D, \xi_D)$, which only depends on
$S(D)$ and has  $Kod=-\infty$.

(iii) If $b^+(Q_D)=0$ but $Q_D$ is not negative definite, then it does not admit a regular neighborhood with  contact boundary.
\end{theorem}


 Golla and Lisca considered  a large family $\cal F$ of torus bundles and showed that these torus bundles are equipped with contact structures arising from 
Looijenga  $D$ with  $b^+(Q_D)= 1$
(Theorem 2.5 in \cite{GoLi14}). They also showed, for  a subfamily of these torus bundles, such a  contact structure is  the unique universally tight  contact structure with vanishing Giroux torsion (Theorem 1.2 in \cite{GoLi14}). This  led them to formulate the following conjecture. 

\begin{conj} [\cite{GoLi14}]
For a concave cycle $D$ of symplectic spheres, the contact structure $\xi_D$ on $Y_D$ is universally tight. 
\end{conj} 


 Moreover,  they investigated Stein (and symplectic) fillings  and  classified in many cases up to diffeomorphism (Theorems  3.1, 3.2, 3.5 in \cite{GoLi14}).
On the other hand, Ohta and Ono classified symplectic fillings of simple elliptic singularities up to symplectic deformation (Theorems 1, 1', 2 in \cite{OhOn03}). 
Using  these results and Corollary \ref{cor: finite deformation},  we establish the following finiteness result.

\begin{cor}[Symplectic fillings, \cite{LiMa17-contact}] \label{cor: finite stein} Suppose $(X, D, \omega)$ is a symplectic log Calabi-Yau  pair with  $b^+(Q_D)=1$.  Then  

$\bullet$ There are finitely many (at least $1$)  Stein fillings of $(Y_D, \xi_D)$ up to symplectic deformation, all having $b^+=0$. Moreover, for a Looijenga pair, all Stein fillings 
have  $c_1=0$.

$\bullet$ This is also true for minimal symplectic fillings.
\end{cor}

 We end the survey
 discussing the geography of Stein fillings for  negative definite $Q_D$.

The first author is grateful for the opportunity to speak at the 
`Perspectives of Mathematics in the 21st Century: Conference in Celebration of the 90th Anniversary of Mathematics Department of Tsinghua University'.
The authors are also grateful to Kaoru Ono for his interest and useful discussions.   
The  authors were supported by NSF grants DMS 1065927 and 1207037, and are supported by NSF grant 1611680. 

\section{Topology of cycle of spheres in a rational surface}
In this section we  review some homological facts about topological divisors, especially cycles of spheres,  and  we refer to  \cite{Ne81}, \cite{GoLi14} and \cite{LiMa17-contact}  for details. 
We first  introduce a pair of basic operations for topological divisors.

\begin{definition} 
Toric blow-up is the operation adding a sphere component  with self-intersection $-1$   between an adjacent pair of components  $C_i$ and $C_{i+1}$ and reducing the self-intersection of $C_i$ and $C_{i+1}$ by $-1$. 
Toric blow-down is the reverse operation. 

Notice that there is a  natural labeling  for these operations. 

Two pairs $(X, D^0)$ and $(X, D^1)$ are said to be toric equivalent if they are connected by toric blow-ups and toric blow-downs. 
$D$ is said to be toric minimal if no component is an exceptional sphere. Here, an exceptional sphere is a  sphere with self-intersection $-1$. 
\end{definition}

They can be performed in the holomorphic and symplectic categories. In the holomorphic category they are often referred as corner blow-up/down. 

\begin{lemma} The  following are preserved under a toric blow-up/down:

$\bullet$ $D$ being a cycle of spheres,

$\bullet$ the non-degeneracy of the intersection matrix $Q_D$,

$\bullet$ the oriented diffeomorphism type of the plumbed 3-manifold $Y_D$.  

\end{lemma} 

The 1st bullet is obvious, while the 2nd bullet is by a direct computation. The 3rd bullet is part of Proposition 2.1 in \cite{Ne81}. 

Here is an example to illustrate how a   sphere with $s=0$ can be  used  to `balance' the self-intersection of the two sides by performing a toric blow-up and a toric blow-down. 

\begin{example}[Toric move]\label{eg: balancing self-intersection by $0$-sphere}
The following  three  cycles of spheres are toric equivalent: 
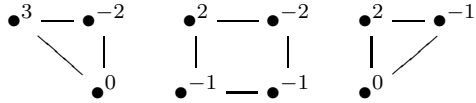
\begin{displaymath}
    \xymatrix @R=1pc @C=1pc {
     \bullet ^{3} \ar@{-}[r] \ar@{-}[dr]& \bullet ^{-2} \ar@{-}[d]        & \bullet ^{2} \ar@{-}[r] \ar@{-}[d]& \bullet ^{-2} \ar@{-}[d] &  \bullet ^{2} \ar@{-}[r] \ar@{-}[d]& \bullet ^{-1} \ar@{-}[dl] \\
	& \bullet ^{0}      &     \bullet ^{-1}  \ar@{-}[r]   & \bullet ^{-1}  &     \bullet ^{0}     \\
}
\end{displaymath}

\end{example}

From now on  $D$ is either a smooth torus or a cycle of smooth spheres. 
When $D$ is a torus with self-intersection $s$, the boundary 3-manifold is the circle bundle with Euler number $s$.



\subsection{The sequence $S(D)$ and  the boundary torus bundle}
When $D$ is a cycle of spheres  the labeling is taken to be cyclic.  The orientation of $D$ is a cyclic labeling up to permutation.  
We will assume now that  $D$ is a cycle of spheres with the self-intersection sequence $S(D)=(s_i)$. 
Let  $s(D)=\sum_{i=1}^{r(D)} (s_i +2)$ denote the self-intersection number of $D$.

\begin{lemma} [cf. Theorem 2.5 and Theorem 3.1 in \cite{GoLi14}]  \label{lem: homology of neighborhood} Let  $D$ be a cycle of spheres in $X$ and $V=X-N_D$. 

$\bullet$   $H_2(N_D)=\Z^{r(D)}=H^2(N_D), H_1(N_D)=H^1(N_D)=\Z,  H_3(N_D)=H^3(N_D)=0$.
  
$\bullet$  $H_1(Y_D)\to H_1(N_D)$ is a surjection.    If  $Q_D$ is non-degenerate, then $b_1(Y_D)=1$ and the map $H_1(Y_D)\to H_1(N_D)$ has a finite kernel, $H_2(Y_D)=H^1(Y_D)=\Z$ and  the map $H_2(Y_D)\to H_2(N_D)$ is trivial.

$\bullet$ Suppose  $Q_D$ is non-degenerate and  $b_1(X)=0$, then $b_1(V)=b_3(V)=0$, $b_2(V)=b_2(X)-r(D)-1$ and the map $\Z=H_2(Y_D)\to H_2(V)$ is injective.

\end{lemma}

Here are obvious restrictions on homologous components of $D$ from the cycle condition. 

\begin{lemma}  \label{lem:   homologous components}  For a cycle of spheres $D$, 

$\bullet$  At most three components   are homologous in $X$. There are three homologous components only if $r(D)=3$.

$\bullet$ There are a pair of homologous components only if  $r(D)\leq 4$.  

$\bullet$ If $[C_i]=[C_{i+1}]$ for some $i$  then $r(D)=3, 
s_i=s_{i+1}=1$, or $r(D)=2, s_i=s_{i+1}=2$. 
\end{lemma}

 When $b^+(X)=1$ there are various restrictions on components with non-negative self-intersection.
 Let $r^{\geq 0}(D)$ denote the number of components with self-intersection $\geq 0$.  

\begin{lemma} \label{lem: non-negative components}
Suppose $D$ is  a cycle of spheres in $X$ with  $b^+(X)=1$. 

$\bullet$   If $C_i$ and $C_j$ are not adjacent and $s_i\geq 0, s_j\geq 0$, then $[C_i]=[C_j]$ and $s_i=s_j=0$. 

$\bullet$ $r^{\geq 0}(D)\leq 4$. 

$\bullet$  $r^{\geq 0}(D)=4$ only if $r(D)=4, s_i=0$ for each $i$ and $[C_1]=[C_3], [C_2]=[C_4]$. 

$\bullet$  Suppose $r(D)\geq 3$. If  $s_i\geq0, s_{i+1}\geq 0, s_is_{i+1}\geq 1$ for some $i$, then $[C_i]=[C_{i+1}]$ and $s_i=s_{i+1}=1$. This is only possible
when $r(D)=3$. 

\end{lemma}

These constraints  follow easily from the $b^+(X)=1$ condition. The following lemma, derived from Lemmas   \ref  {lem:   homologous components} and  \ref {lem: non-negative components}, is very useful for Theorems  \ref{theorem: anti-canonical}, \ref {prop: convex-concave} and  \ref {cor: finite stein}.

\begin{lemma} [\cite{LiMa17-contact}] \label{lem: |D| > 4 => D at most two consecutive nonnegative spheres}
Suppose $D$ is  a cycle of spheres in $X$ with  $b^+(X)=1$.  Up to cyclic permutation and orientation of $D$, we have

$\bullet$  If $r(D)\geq 5$, then $r^{\geq 0}(D)\leq 2$. When $r^{\geq 0}(D)=2$,    $s_1\geq 0, s_2=0$. 

$\bullet$ If $r(D)=4$ and  $r^{\geq 0}(D)\geq 3$, then  $S(D)=(k\geq 0, 0, l<0, 0), [C_2]=[C_4], l+k\leq 0$.

$\bullet$  If  $r(D)=4$ and $r^{\geq 0}(D)=2$, then   either $S(D)=(0, l_1<0, 0, l_2<0), [C_1]=[C_3]$ or $(s_i)=(k\geq 0, 0, l_1<0, l_2<0), l_1+l_2+k\leq 0$. 

$\bullet$ If  $r(D)=3$ and $r^{\geq 0}(D)=3$, then the only possibilities of $S(D)$ are (i) $(1, 1, 1), [C_1]=[C_2]=[C_3]$,  (ii) $(1, 1, 0), [C_1]=[C_2]$, (iii)  $(2\geq k\geq 0, 0, 0)$.

$\bullet$ If $r(D)=3$ and $r^{\geq 0}(D)=2$, then the only possibilities of  $S(D)$ are (i)  $(1, 1, p<0), [C_1]=[C_2]$,  (ii)  $(k\geq 0 , 0, p<0), p+k\leq 2$.

$\bullet$ If  $r(D)=2$ and  $r^{\geq 0}(D)=2$, then  the only possibilities of $S(D)$ are  $(4, 1), (4, 0), (3, 1),\\ (3,0),  (2, 2), (2, 1), (2, 0), 
 (1, 1), (1, 0), (0, 0)$.

$\bullet$ If  $r(D)=2$ and  $r^{\geq 0}(D)=1$, then  $S(D)=(k\geq 0, p<0)$.

$\bullet$ If  $r(D)=2$ and  $r^{\geq 0}(D)=0$, then   $S(D)$ is one of $(-1, -1),  (-1, -2), (-1, -3) $.

\end{lemma}

To describe the plumbed 3-manifold $Y_D$, we introduce the matrix in  $SL_2(\mathbb Z)$ for a sequence of integers
$(-t_1, \cdots, -t_k)$,
\[ A(-t_1,\dots,-t_k)= \begin{pmatrix}
-t_k & 1 \\
-1 & 0 \end{pmatrix}
\begin{pmatrix}
-t_{k-1} & 1 \\
-1 & 0 \end{pmatrix} \dots 
\begin{pmatrix}
-t_1 & 1 \\
-1 & 0 \end{pmatrix}.\]

\begin{lemma}  [Theorem 6.1 in \cite{Ne81}, Theorem 2.5 in \cite{GoLi14}]   \label{lem: property of continuous fraction}
For a cycle of spheres $D$ with self-intersection sequence $S(D)=(s_1, ..., s_k)$, the  plumbed 3-manifold $Y_D$ is the oriented torus bundle $T_A$ 
over $S^1$ with  monodromy  $A=A(-s_1,\dots,-s_k)$.
The intersection matrix $Q_D$ is non-degenerate   if the trace of $A(-s_1,\dots,-s_k)\ne 2$.   

\end{lemma}

\subsection{Toric minimal pairs}

 \begin{lemma}\label{lem: not negative semi-definite => at least one non-negative}
Any cycle of sphere is toric equivalent to a toric minimal one or one with sequence $(-1, p)$. If $S(D)=(-1,p)$, then $Q_D$ is degenerate only if $p=-4$.

Suppose $D$ is  a toric minimal cycle of spheres with  sequence $S(D)=(s_i)$. Then

$\bullet$  $b^+(Q_D)\geq 1$    if and only if  $s_i\geq 0$ for some $i$.

$\bullet$  $Q_D$ is  negative definite if 
$s_i\leq -2$   for all $i$ and less than $-2$ for some $i$.
   $Q_D$ is  negative semi-definite but not    negative definite if 
$s_i= -2$ for each $i$. 


$\bullet$    $Q_D$ is non-degenerate 
 if either  $s_1\geq 0$ and $s_i\leq -2$ for $i\geq 2$, or    $s_1=s_2=0$ and $s_i\leq -2$ for $i\geq 3$, 
\end{lemma}

The first statement is by definitions (Notice that we do not allow nodal components). The second statement is obvious. 
Bullets 1, 2 are well-known (cf. Lemma 8.1 in \cite{Ne81}).  To  prove the 3rd bullet,  by Lemma \ref{lem: property of continuous fraction}, 
we just need to  that the trace of the monodromy matrix  is not equal to $2$, which is a direct calculation using  Lemma 5.2 in \cite{Ne81}.

  Each  toric minimal, negative definite cycle $D$ with $s(D)\leq -2$ has a dual cycle $\check D$, with the property that the plumbed manifolds $Y_D$ and $Y_{\check D}$ are orientation reversing diffeomorphic (Theorem 7.1 in \cite{Ne81}). 
   To describe the dual cycle we use the $2$ by $k$ matrix $\begin{pmatrix}a_1 &\dots & a_k\cr b_1 &\dots &b_k \end{pmatrix}$ to represent  the sequence 
$(a_1,-2,\dots,-2,a_2,\dots,a_k),$
where $a_i\leq -3$ and there are $b_i$ many $-2$ between $a_i$ and $a_{i+1}$.
For a negative definite toric minimal cycle $D$ with $s(D)\leq -2$, we have either two $a_i$ terms or $a_i\leq -4$ for some $i$.
The dual cycle $\check{D}$ is  represented by the $2$ by $k$ matrix $\begin{pmatrix} \check{a}_i&=&-b_i-3 \cr \check{b}_i&=&-a_{i+1}-3   \end{pmatrix}$. It is easy to check  that $\check D$ is also toric minimal, negative definite and $s(\check D)\leq -2$.  
A remark is that we can also view the elliptic pairs $(s)$ and $(-s)$ as dual pairs in the sense that boundary 3-manifolds are  orientation reversing diffeomorphic.

\section{Algebraic geometry of Looijenga pairs}
In this section we very briefly review  some basic results  of Looijenga pairs $(Y, D)$, which have or might have symplectic analogues.   
Please consult the survey article  \cite{Fr} and \cite{GrHaKe11}.

\subsection{Torelli and deformation}

There are several versions of the Torelli theorem. The following is Theorem 8.5 in \cite{Fr}. 

\begin{theorem} [A global Torelli]
Given Looijenga pairs and an isomorphism of lattices $\mu$ compatible with $D$, there is a isomorphism $f$ of Looijenga pairs such that $\mu=f^*$ if and only if $\mu$ preserves the nef cone. 
\end{theorem}

Two anticanonical pairs are said to be (holomorphically) deformation equivalent if they are both isomorphic to fibers of a family of anticanonical pairs
over a connected base. 
The following two statements are given in Theorem 3.1 and Theorem 5.14 in \cite{Fr} respectively.

\begin{theorem} \label{thm: finite kahler deformation}
 There are only finitely many deformation types of Looijenga pairs with the same self-intersection sequence. 
Two Looijenga pairs are deformation equivalent if they are homology equivalent. 
\end{theorem}






\subsection{Cusp singularities}

A cusp singularity is the germ of an isolated, normal surface singularity such that the exceptional divisor of the minimal resolution
is a cycle of smooth rational curves  $D$ meeting transversely. 
For normal surface singularities, there is a notion of Kodaira dimension $\kappa^{\delta}$, 
and Gorenstein surface singularities  with $\kappa^{\delta}=0$ are simple elliptic singularities and   cusp singularities   (cf. \cite{OhOn09} and the references therein). 

Cusp singularities come in dual pairs, and their minimal resolutions are given as dual cycles. Every pair of dual cycles embed in a Hirzebruch-Ionue surface as the only curves. 
A cusp singularity is called rational if its minimal resolution is realized as the anti-canonical divisor of a rational surface. 
By the Mumford-Grauert criterion, any toric minimal, negative definite Looijenga pair  $(Y, D)$ arises as the minimal resolution of a rational cusp singularity. 
Looijenga proved that  a cusp is rational if its dual cusp is smoothable and he conjectured the converse is also true. 
The Looijenga conjecture was proved in \cite{GrHaKe11} via mirror symmetry and later by integral-affine geometry in \cite{En}. 


\section{Deformation classes of symplectic log CY pairs}

In this section we give a  brief outline of the proof of Theorem \ref {thm: symplectic deformation class=homology classes} and Theorem \ref{theorem: anti-canonical}.

\subsection{Operations and minimal pairs}

It involves the operations of non-toric blow-up/down  and the notion of minimal models.  
A {\it non-toric blow-up} of $D$ is the  proper transform of a symplectic blow-up centered  at a smooth point of $D$.
A  non-toric blow-down is the reverse operation which symplectically blows down an exceptional sphere  not contained in $D$.
These operations preserve the log Calabi-Yau condition and there are analogues in the holomorphic category, sometimes referred as  interior blow-up/blow-down.

A  symplectic log Calabi-Yau pair  $(X,D,\omega)$ is called 
 minimal   if $(X, \omega)$ is minimal,  or 
   $(X, D, \omega)$ is a symplectic  Looijenga pair
with $X=\mathbb{C}P^2 \# \overline{\mathbb{C}P^2}$. 
For any  symplectic log Calabi-Yau pair  $(X,D,\omega)$, we  apply  first a maximal sequence of non-toric blow-downs using \cite{McOp13} and then a maximal sequence of toric blow-downs. The resulting  toric minimal pair, which  is  actually minimal due to  \cite{Pi08},   is called a minimal model of $(X,D,\omega)$.

We  enumerate the minimal symplectic log Calabi-Yau pairs (modulo cyclic symmetry), all of them having length less than $5$. 
 
 

$\bullet$  Case $(A)$: The base genus of $X$ is $1$. $D$ is a torus. 

 
 
 

$\bullet$  Case $(B)$: $X=\mathbb{C}P^2$,
$c_1=3h$. 

$(B1)$ $D$ is a torus, 

$(B2)$ $D$ consists of a $h-$sphere and a $2h-$sphere, or

$(B3)$ $D$ consists of three $h-$spheres.

$\bullet$ Case $(C)$: $X=\mathbb{S}^2 \times \mathbb{S}^2$, $c_1=2f_1+2f_2$, where $f_1$ and $f_2$ are the homology classes
of the two factors. 

$(C1)$ $D$ is a torus. 

$(C2)$ $r(D)=2$ and  $[C_1]=bf_1+f_2, [C_2]=(2-b)f_1+f_2$.

$(C3)$ $r(D)=3$ and  $[C_1]=bf_1+f_2, [C_2]=f_2, [C_3]=(1-b)f_1+f_2$.

$(C4)$ $r(D)=4$ and $[C_1]=bf_1+f_2,  [C_2]=f_1, [C_3]=-bf_1+f_2, [C_4]=f_1$.

The graphs in (C1), (C2), (C3) and (C4) are given respectively  by
$$      
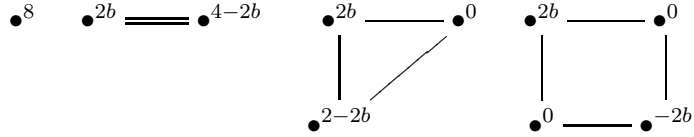
\xymatrix{
       \bullet^{8} 
	} 
	\quad  
 \xymatrix{
       \bullet^{2b} \ar@{=}[r] & \bullet^{4-2b} 
	} 
	\quad 
	\xymatrix{
       \bullet^{2b} \ar@{-}[d] \ar@{-}[r] & \bullet^{0} \ar@{-}[dl] \\
       \bullet^{2-2b}
	} 
	\quad   
	\xymatrix{
       \bullet^{2b} \ar@{-}[d] \ar@{-}[r] & \bullet^{0} \ar@{-}[d] \\
       \bullet^{0} \ar@{-}[r] & \bullet^{-2b} \\}
       $$

$\bullet$ Case $(D)$: $X=\mathbb{C}P^2 \# \overline{\mathbb{C}P^2}$, $c_1=f+2s$, where $f$ and $s$ are the fiber class and section class 
with $f\cdot f=0$, $f\cdot s=1$ and $s\cdot s=1$.


$(D2)$ $r(D)=2$,  and either 
$([C_1],[C_2])=(af+s,(1-a)f+s)$ or $([C_1],[C_2])=(2s, f)$.

$(D3)$  $r(D)=3$ and  $[C_1]=af+s, [C_2]=f, [C_3]=-af+s$.

$(D4)$ $r(D)=4$ and  $[C_1]=af+s, [C_2]=f,  [C_3]=-(a+1)f+s, [C_4]=f$.

The graphs in (D2), (D3) and (D4) are given respectively by
$$    
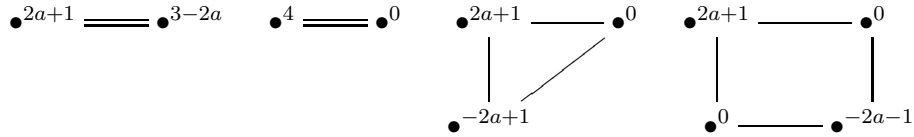
\xymatrix{
       \bullet^{2a+1} \ar@{=}[r] & \bullet^{3-2a} }
       \quad
        \xymatrix{
      \bullet^{4} \ar@{=}[r] & \bullet^{0} }
      \quad
       \xymatrix{
       \bullet^{2a+1} \ar@{-}[d] \ar@{-}[r] & \bullet^{0} \ar@{-}[dl] \\
       \bullet^{-2a+1}  \\
	}
	\quad
	 \xymatrix{
       \bullet^{2a+1} \ar@{-}[d] \ar@{-}[r] & \bullet^{0} \ar@{-}[d] \\
       \bullet^{0} \ar@{-}[r] & \bullet^{-2a-1} \\
	}	
       $$

\subsection{Classification by homology equivalence}
There are two steps to prove Theorem \ref {thm: symplectic deformation class=homology classes}. One step is to 
show that each (strict) homology type of  minimal pairs contains a unique (strict) deformation class via a combination of pseudo-holomorphic curve techniques and 
Thurston type symplectic construction in the setting of a pair of a symplectic 4-manifold with a smooth symplectic surface.

We also introduce marked divisors and establish the invariance of their (strict) deformation class under toric and non-toric blow-up/down operations (cf. also \cite{OhOn03}). This invariance property reduces Theorem \ref{thm: symplectic deformation class=homology classes} to the minimal case. 
The statement that each symplectic deformation class contains a K\"ahler   pair is not stated in \cite{LiMa16-deformation} but it follows  from the proof outlined above since 
each minimal pair clearly deforms to  a K\"ahler pair (cf. Section 3 in \cite{LiMa16-deformation} and Theorem 2.4 in \cite{Fr}) and blow-up/down can be performed in the K\"ahler category. 

We remark that  Theorem \ref {thm: symplectic deformation class=homology classes} should also apply  to  the cases of irreducible nodal spheres and cuspidal spheres using \cite{Ba} and \cite{OhOn05-cuspidal} respectively. 

\begin{proof} [Proof of Corollary \ref{cor: finite deformation}]
By Theorem \ref {thm: symplectic deformation class=homology classes}, every  symplectic deformation class contains a K\"ahler pair.  The finiteness of Looijenga pairs follows directly from Theorem  \ref{thm: finite kahler deformation}. For elliptic symplectic log Calabi-Yau pairs, where the sequences are of length $1$,   the finiteness  is more straightforward--it follows from the
finiteness of symplectic deformation types in  the case of minimal pairs for each  $(s)$, where $s=0, 8, 9$ (cf. Section 3 in \cite{LiMa16-deformation}),  and the fact that there is only one way to (non-toric) blow up, up to deformation.  \end{proof}


\subsection{Anti-canonical sequences}

Due to the classification of minimal symplectic log Calabi-Yau pairs, it is a combinatorial problem to determine the anti-canonical sequences. 
There are also  various conditions on spherical circular sequences with $b^+=1$  in   Lemma \ref{lem: not negative semi-definite => at least one non-negative}, Lemma   \ref{lem: |D| > 4 => D at most two consecutive nonnegative spheres},  Lemma  \ref{lem: non-negative components}. The first statement of Theorem \ref{theorem: anti-canonical} that every spherical circular sequence with $b^+=1$ is anti-canonical is deduced from these lemmas, the list of minimal pairs, the observation that whether a spherical circular sequence is anti-canonical only depends on its toric equivalence class, 
and




 
\begin{prop}
Suppose $D\subset (X, \omega)$ is a cycle of spheres in a rational surface $(X, \omega)$ with minimal complement. Then
 $s(D)\leq 9$, and  $S(D)\ne (5+l, -l)$ with $l\geq 2$. 

$D$ represents $c_1 (X, \omega)$ if 

$\bullet$  $s_i\geq -1$ for any $i$, or  

$\bullet$ $S(D)=(1, -p_1+1, -p_2, ...., -p_{l-1}, -p_l+1)$ with $p_i\geq 2$, $l\geq 2$. 
\end{prop}

This proposition is proved using Theorem 6.10 in \cite{OhOn03}, Proposition 3.14 in \cite{LiZh11}, Theorem 3.1 in \cite{GoLi14}, and a direct verification to exclude $(5+l, -l)$ with $l\geq 0$.


For the second statement of Theorem \ref{theorem: anti-canonical} that   any  anti-canonical sequence  with $b^+=1$ is rigid,
it follows from the following propositions and  the observation that whether an anti-canonical sequence
is  rigid 
only depends on its toric equivalence class. 


 \begin{prop}\label{lem: positive parabolic} 
Suppose $(s_i)$ is an  anti-canonical sequence
and it belongs to one in the following list. 

 $\bullet$ $(1, -p_1+1, -p_2, ...., -p_{l-1}, -p_l+1)$ with $p_i\geq 2$, $l\geq 2$ so $r(D)\geq 3$.  
 
 $\bullet$ $(0, 0, 0, n)$ with $n\leq 0$.

$\bullet$  $(1, 1, p), p\leq 1$.
  

 $\bullet$ $(1, p)$ with $p\geq 4$.  
  
 $\bullet$  $(0, n)$ with $n\leq 4$.
 
 
 $\bullet$   $s_i\geq -1$ for each $i$. 
 
 $\bullet$  $(-1, -2)$ and $(-1, -3)$.
 
  Then $(s_i)$ is rigid. 
 
 
 \end{prop}
 
 \begin{prop} \label{lem: rigid} Suppose $(X, D, \omega)$ is a symplectic Looijenga pair with $b^+(Q_D)= 1$. Then $S(D)$ is toric equivalent to one in Proposition \ref{lem: positive parabolic}. 
\end{prop}
 
 Proposition \ref{lem: positive parabolic}, except for the last bullet,  is proved using  Proposition 7.1 in \cite{OhOn03}, 
 Theorems 3.1, 3.2, 3.5 in \cite{GoLi14} and similar arguments. The cases $(s_i)=(-1, -2)$ and $(-1, -3)$ are more delicate, requiring a blowup trick. 
 Proposition \ref {lem: rigid} is proved by Lemmas \ref{lem: not negative semi-definite => at least one non-negative},   \ref{lem: |D| > 4 => D at most two consecutive nonnegative spheres},   \ref{lem: non-negative components},  the toric move in Example \ref{eg: balancing self-intersection by $0$-sphere} and induction on the length of $D$.  

\section{Contact aspects}

Let $(X, D, \omega)$ be a symplectic log Calabi-Yau pair. 
A  neighborhood $N'$ of $D$ is called a  concave (resp. convex) neigborhood if $N'$ is a  concave (resp. convex)  symplectic manifold.
$D$ is called concave  (resp. convex) if   for any neighborhood $N'$ of $D$, there is a concave (resp. convex) plumbing neighborhood $N_D \subset N'$.   A necessary condition for $D$ to be either convex or concave is $\omega$ being exact on the boundary of any plumbing neighborhood. 
Here is a  local criterion. 

\begin{lemma}
\label{lem: non-degenerate intersection form}  
$\omega|_{Y_D}$ is exact if and only if there is a solution for $z$ to the equation $Q_Dz=a$, where $a=([\omega]\cdot [C_1],\dots,[\omega]\cdot [C_k])$ is the area vector. 
In particular, this holds if  $Q_D$ is non-degenerate.
Moreover,  this condition only depends on the toric equivalence class. 
\end{lemma}

The first statement  is observed in  \cite{LiMa14}. 
Moreover, tori blow-up/down is a local operation that does not change the the diffeomorphism type of $Y_D$ and the exactness of $\omega|_{Y_D}$. 
One can also check that the solvability  for $Q_Dz=a$ is stable under toric blow-up/down by simple linear algebra. 
When $X$ is a closed manifold, we also have 
the following criterion.

\begin{lemma}\label{lem: orthogonal decomposition}
Suppose $X$ is a closed manifold with intersection matrix $Q_X$.
Let $I_1=\iota_*(H_2(D);\R) \subset H_2(X;\R)$ and $I_2 \subset H_2(X;\R)$ be $Q_X$-orthogonal to $I_1$ in $H_2(X;\R)$.
If the span of $I_1 \cup I_2$ is $H_2(X;\R)$, then $\omega|_{Y_D}$ is exact. 
The existence of $I_2$ is preserved under toric blow-up and toric blow-down.

\end{lemma}

We also recall  two criterions for symplectic divisors to be contact and the definition of contact Kodaira dimension. 

\begin{theorem}[\cite{GaSt09}, \cite{McL12}] \label{GS} A negative definite symplectic divisor is convex. 
\end{theorem}

\begin{theorem}  [\cite{LiMa14}] \label{concave}
Let $D \subset (W,\omega_0)$ be a symplectic divisor.
If  $Q_D$ is not negative definite and
 $\omega_0$ restricted to the boundary  of $D$ is exact, 
then $\omega_0$ can be locally deformed through a family of symplectic forms $\omega_t$ on $W$ keeping $D$  symplectic   and such that $(D,\omega_1)$ is a concave divisor.
Moreover, the contact structure $\xi_D$ on $Y_D$  is canonically associated to $D$ in this case and in the negative definite case. 
\end{theorem}

\begin{definition}[\cite{LiMa16}, \cite{LiMaYa14}]
Let $(W, \omega)$ be a  concave symplectic 4-manifold with contact boundary $(Y, \xi)$. $(W, \omega)$  is called a Calabi-Yau  cap of $(Y, \xi)$ if $c_1(W)$ is a torsion class,  
and it is called a uniruled cap of $(Y, \xi)$ if there is a contact primitive $\beta$ on the boundary such that $c_1(W)\cdot  [(\omega, \beta)]>0$.

The contact Kodaira  dimension   of a contact 3-manifold $(Y, \xi)$ is  defined in terms of uniruled caps and Calabi-Yau caps. 
Precisely, $Kod(Y, \xi)=-\infty$ if $(Y, \xi)$  has a uniruled cap, $Kod(Y, \xi)=0$ if it has a Calabi-Yau cap but no uniruled caps, $Kod(Y, \xi)=1$ if it has no Calabi-Yau caps or uniruled caps. 

\end{definition}

\subsection{Trichotomy}

Theorem \ref{prop: convex-concave} is based on the following observation in \cite{LiMa17-contact} (cf. also Theorem 2.5 in \cite{GoLi14}).

\begin{prop}\label{prop: exact}   For  a symplectic log Calabi-Yau pair  $(X, D, \omega)$,  $\omega$  is exact on $Y_D$
if and only if   $Q_D$ is  negative definite or $b^+(Q_D)=1$.

\end{prop}

This result is proved by the local criterion Lemma \ref{lem: non-degenerate intersection form},  Lemma \ref{lem: not negative semi-definite => at least one non-negative}, Lemma   \ref{lem: |D| > 4 => D at most two consecutive nonnegative spheres},  Lemma  \ref{lem: non-negative components}, the toric move in Example \ref{eg: balancing self-intersection by $0$-sphere}, and by applying the $I_2-$criterion Lemma \ref {lem: orthogonal decomposition} to the following list of log Calabi-Yau pairs $(X, D, \omega)$ with $r(D)\leq 4$ and 
$b^+(Q_D)=1$. 
\begin{enumerate}

\item (B2) in the list of minimal models;  $I_2=\emptyset$; $S(D)=(1, 4)$.

\item (C2) with $b=1$;  $I_2=\emptyset$; $S(D)=(2, 2)$.

\item (B3);  $I_2=\emptyset$; $S(D)=(1, 1, 1)$.

\item  Non-toric blow-ups of (B3) on $C_3$ and its proper transforms;  $I_2=\{e_j-e_{j+1}, 1\leq j\leq \alpha-1\}$; $S(D)=(1, 1, 1-\alpha)$.

\item  Non-toric blow-ups of (C3) on $C_3$ and its proper transforms;   $I_2=\{e_j-e_{j+1}, 1\leq j\leq \alpha-1\}$; $S(D)=(0, 0, 2-\alpha)$. 

\item  (C4) with $b=0$;  $I_2=\emptyset$; $S(D)=(0, 0, 0, 0)$.  

\item   Non-toric blow-ups of (C4) with $b=0$ on $C_4$ and its proper transforms;  $I_2=\{e_j-e_{j+1}, 1\leq j\leq \alpha-1\}$; $S(D)=(0, 0, 0, -\alpha)$.

\end{enumerate}

For Case (iii) of Theorem \ref{prop: convex-concave}, it follows from Proposition \ref{prop: exact} that $\omega$ is not exact on $Y_D$. 
For Case (i) of Theorem \ref{prop: convex-concave},   $Q_D$ is negative definite and hence there is a  convex plumbing neighborhood $N_D$ with contact boundary $(Y_D, \xi_D)$ by  Theorem \ref{GS}.
Notice that  $P=X-N_D$ is a symplectic cap of $Y_D$ with vanishing $c_1$,
namely, it  is a Calabi-Yau cap. It follows that $Kod(Y_D, \xi_D)\leq 0$. 
For Case (ii) of Theorem \ref{prop: convex-concave},  it follows from Theorem \ref{concave} and Proposition \ref{prop: exact} that, up to a local symplectic deformation,  there is a  concave plumbing neighborhood $N_D$ with contact boundary $(Y_D, \xi_D)$.  
Moreover, since $D$ is symplectic and represents $c_1(X)$, for any contact primitive $\alpha$ of $\omega|_{Y_D}$,  we have 
$c_1(N_D)\cdot [(\omega,\alpha)]=   c_1(X)|_{N_D} \cdot [(\omega,\alpha)] = D\cdot  [(\omega, \alpha)]=D\cdot [\omega]>0$. 
Thus $N_D$ is a uniruled cap. 



\begin{remark}\label{thm: as a support of ample line bundle}
Applying  Theorem \ref{prop: convex-concave}, Theorem  \ref{thm: symplectic deformation class=homology classes} 
and Proposition 4.1 in \cite{GrHaKe12}, it is not hard to prove the following statement: 
For  a symplectic log Calabi-Yau pair $(X, D, \omega)$ with 
$b^+(Q_D)= 1$,     there exists a K\"ahler log Calabi-Yau pair $(\overline{X},\overline{D},\overline{\omega})$ in its  symplectic deformation  class
such that $\overline{D}$ is the support of an ample line bundle. 
Then  $(\overline{X}-\overline{D},\overline{\omega})$ provides a Stein filling with $b^+=0$ and $c_1=0$. 
\end{remark}

\subsection{Symplectic fillings}


In the context of torus bundles,  Golla-Lisca investigated symplectic fillings in  the case     $b^+(Q_D)= 1$. Here is a summary of their results.

\begin{theorem} [Theorems 1.1, 3.1, 3.5  in \cite{GoLi14}]\label{thm: GoLi}  
For a large family $\cal F$ of torus bundles $T_A$ arising from $D$ with $b^+(Q_D)= 1$, all Stein fillings of $(T_A=Y_D, \xi_D)$ have $c_1=0$, $b_1=0$ and the same $b_2$.  
Moreover, up to diffeomorphism, there are
only finitely many Stein fillings, and there is a unique 
 Stein filling if $|tr A|<2$. Here $A$ is the monodromy matrix of $Y_D$. 
These results also hold for minimal symplectic fillings for this family, except possibly 3 torus bundles 
 with $|tr A|<2$.
 \end{theorem}

According to Corollary \ref{cor: finite stein}, the finiteness property holds more generally. 



\begin{proof}[Proof of Corollary \ref{cor: finite stein}]  
By  Theorem  \ref{prop: convex-concave} $(Y_D, \xi_D)$ is fillable and all the symplectic fillings have $b^+=0$. 
For  Looijenga pairs, the Stein filliability follows from Remark    \ref{thm: as a support of ample line bundle}. For an elliptic pair with self-intersection $s>0$, there is an obvious
Stein filling diffeomorphic to the neighborhood of a torus with self-intersection $-s$. The finiteness of symplectic fillings for elliptic pairs  is proved in \cite{OhOn03} (see Theorem \ref{OO-simple elliptic}).

Now observe that if $D$ is  concave and (Stein) rigid  then  any (Stein) symplectic  filling of $(Y_D, \xi_D)$ is the complement of a symplectic log CY pair with the same self-intersection  sequence. 
Now we invoke   the second statement of Theorem \ref{theorem: anti-canonical} and  Corollary \ref{cor: finite deformation} to conclude the finiteness of Stein symplectic fillings for all Looijenga pairs and the finiteness of symplectic fillings except  for the toric equivalence classes of $(-1, -2), (-1, -3)$.  
Clearly, the fillings have vanishing $c_1$.

\end{proof}

Together with  Theorems  1.3 and 1.8 in \cite{LiMaYa14}, Theorem \ref{prop: convex-concave} has the following consequence:
when $Q_D$ is negative definite, the Betti numbers of exact fillings of $(Y_D, \xi_D)$ are bounded.  For elliptic pairs, we have the following:

\begin{theorem}  [Theorem 2 in \cite{OhOn03}] \label{OO-simple elliptic} Any simple elliptic singularity  has  finite number of symplectic   fillings,   arising either from  a smoothing or the minimal resolution .  
\end{theorem}

For Looijenga pairs, when $D$ is negative definite and toric minimal, $\xi_D$ coincides with the contact structure arising from the corresponding cusp singularity  and hence is Stein fillable with a Stein filling diffeomorphic to $N_D$. Notice that  $b_1(N_D)=1$ by Lemma \ref{lem: homology of neighborhood}.
 We provide some explicit Betti number bounds for Stein fillings below when $D$ is negative definite. 

\begin{prop}[\cite{LiMa17-contact}]\label{p:AfterGlue}
Suppose that  $D$ is toric minimal and negative definite and $V=X - N_D$.
If $U$ is a Stein filling of $Y_D$, then $X_U=U \cup V$ has either $b^+=1$ or $3$, and  $b^+(X)=1+b^+(U)+b_2^0(U), 
b_2^0(U)+b_1(U)=1$. 

When $b^+(X_U)=1$,  $X_U$ is  rational or an integral homology Enriques surface, and $U$ is negative definite with $b_1(U)=1$. In this case $e(U)=b^-(U)$, where $e$ is the Euler number. 

When $b^+(X_U)=3$, $X_U$ is an integral homology $K3$, $(b_2^+(U),b_2^0(U),b_1(U))=(1,1,0)$ or $(2,0,1)$.
In either case, $c_1(U)=0$ and $2\leq e(U)\leq 21$. 
\end{prop}

Finally, we discuss the potential implication of Proposition \ref{p:AfterGlue} for Stein fillings of   cusp singularities. 
By the now confirmed Looijenga conjecture which states that  a cusp singularity is smoothable if and only if has a rational dual,  a smoothing of a cusp singularity provides a Stein filling with $b^+=1$.  In light of this,  Proposition \ref{p:AfterGlue} provides some evidence to 
the following  symplectic/contact  analogue of the Looijenga conjecture. 

\begin{speculation}
If a  cusp  singularity  does not have  a rational dual, then  it admits only negative definite Stein fillings.
\end{speculation}


\bibliography{May2018}

\begin{thebibliography}{10}

\bibitem{Auroux}
Denis Auroux.
\newblock Mirror symmetry and {$T$}-duality in the complement of an
  anticanonical divisor.
\newblock {\em J. G\"okova Geom. Topol. GGT}, 1:51--91, 2007.

\bibitem{Ba}
Jean-Fran\c{c}ois Barraud.
\newblock Nodal symplectic spheres in {$\bold C{\rm P}^2$} with positive
  self-intersection.
\newblock {\em Internat. Math. Res. Notices}, (9):495--508, 1999.

\bibitem{En}
Philip~Milton Engel.
\newblock {\em A {P}roof of {L}ooijenga's {C}onjecture via {I}ntegral-{A}ffine
  {G}eometry}.
\newblock ProQuest LLC, Ann Arbor, MI, 2015.
\newblock Thesis (Ph.D.)--Columbia University.

\bibitem{Fr}
Robert Friedman.
\newblock On the geometry of anticanonical pairs.
\newblock {\em arXiv:1502.02560}, 2015.

\bibitem{GaSt09}
David~T. Gay and Andr\'as~I. Stipsicz.
\newblock Symplectic surgeries and normal surface singularities.
\newblock {\em Algebr. Geom. Topol.}, 9(4):2203--2223, 2009.

\bibitem{GoLi14}
Marco Golla and Paolo Lisca.
\newblock On {S}tein fillings of contact torus bundles.
\newblock {\em Bull. Lond. Math. Soc.}, 48(1):19--37, 2016.

\bibitem{GrHaKe11}
Mark Gross, Paul Hacking, and Sean Keel.
\newblock Mirror symmetry for log {C}alabi-{Y}au surfaces {I}.
\newblock {\em Publ. Math. Inst. Hautes \'Etudes Sci.}, 122:65--168, 2015.

\bibitem{GrHaKe12}
Mark Gross, Paul Hacking, and Sean Keel.
\newblock Moduli of surfaces with an anti-canonical cycle.
\newblock {\em Compos. Math.}, 151(2):265--291, 2015.

\bibitem{GHKK}
Mark Gross, Paul Hacking, Sean Keel, and Maxim Kontsevich.
\newblock Canonical bases for cluster algebras.
\newblock {\em J. Amer. Math. Soc.}, 31(2):497--608, 2018.

\bibitem{GHS}
Mark Gross, Paul Hacking, and Bernd Siebert.
\newblock Theta functions on varieties with effective anti-canonical class.
\newblock {\em arXiv:1601.07081}, 2016.

\bibitem{LiMa17-contact}
Tian-Jun Li and Cheuk~Yu Mak.
\newblock Symplectic log calabi-yau surface--contact aspects.
\newblock {\em In preparation.}

\bibitem{LiMa14}
Tian-Jun Li and Cheuk~Yu Mak.
\newblock Symplectic divisorial capping in dimension 4.
\newblock {\em arXiv:1407.0564}, 2014.

\bibitem{LiMa16}
Tian-Jun Li and Cheuk~Yu Mak.
\newblock The kodaira dimension of contact 3-manifolds and geography of
  symplectic fillings.
\newblock {\em arXiv:1610.06870}, 2016.

\bibitem{LiMa16-deformation}
Tian-Jun Li and Cheuk~Yu Mak.
\newblock Symplectic log {C}alabi-{Y}au surface---deformation class.
\newblock {\em Adv. Theor. Math. Phys.}, 20(2):351--379, 2016.

\bibitem{LiMaYa14}
Tian-Jun Li, Cheuk~Yu Mak, and Kouichi Yasui.
\newblock Calabi-{Y}au caps, uniruled caps and symplectic fillings.
\newblock {\em Proc. Lond. Math. Soc. (3)}, 114(1):159--187, 2017.

\bibitem{LiZh11}
Tian-Jun Li and Weiyi Zhang.
\newblock Additivity and relative {K}odaira dimensions.
\newblock In {\em Geometry and analysis. {N}o. 2}, volume~18 of {\em Adv. Lect.
  Math. (ALM)}, pages 103--135. Int. Press, Somerville, MA, 2011.

\bibitem{Liu96}
Ai-Ko Liu.
\newblock Some new applications of general wall crossing formula, {G}ompf's
  conjecture and its applications.
\newblock {\em Math. Res. Lett.}, 3(5):569--585, 1996.

\bibitem{McOp13}
Dusa McDuff and Emmanuel Opshtein.
\newblock Nongeneric {$J$}-holomorphic curves and singular inflation.
\newblock {\em Algebr. Geom. Topol.}, 15(1):231--286, 2015.

\bibitem{McL12}
Mark McLean.
\newblock The growth rate of symplectic homology and affine varieties.
\newblock {\em Geom. Funct. Anal.}, 22(2):369--442, 2012.

\bibitem{Ne81}
Walter~D. Neumann.
\newblock A calculus for plumbing applied to the topology of complex surface
  singularities and degenerating complex curves.
\newblock {\em Trans. Amer. Math. Soc.}, 268(2):299--344, 1981.

\bibitem{OhOn96}
Hiroshi Ohta and Kaoru Ono.
\newblock Notes on symplectic {$4$}-manifolds with {$b^+_2=1$}. {II}.
\newblock {\em Internat. J. Math.}, 7(6):755--770, 1996.

\bibitem{OhOn03}
Hiroshi Ohta and Kaoru Ono.
\newblock Symplectic fillings of the link of simple elliptic singularities.
\newblock {\em J. Reine Angew. Math.}, 565:183--205, 2003.

\bibitem{OhOn05-cuspidal}
Hiroshi Ohta and Kaoru Ono.
\newblock Symplectic 4-manifolds containing singular rational curves with
  {$(2,3)$}-cusp.
\newblock In {\em Singularit\'es {F}ranco-{J}aponaises}, volume~10 of {\em
  S\'emin. Congr.}, pages 233--241. Soc. Math. France, Paris, 2005.

\bibitem{OhOn09}
Hiroshi Ohta and Kaoru Ono.
\newblock An inequality for symplectic fillings of the link of a hypersurface
  {$K3$} singularity.
\newblock In {\em Algebraic topology---old and new}, volume~85 of {\em Banach
  Center Publ.}, pages 93--100. Polish Acad. Sci. Inst. Math., Warsaw, 2009.

\bibitem{Pa13}
James Pascaleff.
\newblock On the symplectic cohomology of log calabi-yau surfaces.
\newblock {\em arXiv:1304.5298}, 2013.

\bibitem{Pi08}
Martin Pinsonnault.
\newblock Maximal compact tori in the {H}amiltonian group of 4-dimensional
  symplectic manifolds.
\newblock {\em J. Mod. Dyn.}, 2(3):431--455, 2008.

\bibitem{Sa13}
Dietmar Salamon.
\newblock Uniqueness of symplectic structures.
\newblock {\em Acta Math. Vietnam.}, 38(1):123--144, 2013.

\end{thebibliography}


\begin{thebibliography}{99}






\bibitem{Ba}
J-F Barraud
\newblock Nodal symplectic spheres in CP2 with positive self-intersection.
\newblock{\em  Internat.
Math. Res. Notices}, (1999) 495-508.


\bibitem{En}
P. Engel
\newblock A proof of Looijenga's conjecture via integral-affine geometry.
\newblock   arXiv: 1409.7676.


\bibitem{Fr}
R. Friedman.
\newblock On the geometry of anticanonical pairs.
\newblock arXiv:1502.02560.

\bibitem{FM83}
R. Friedman and R. Miranda.
\newblock Smoothing cusp singularities of small length.
\newblock {\em Math. Ann.} 263 (1983), 185-212.

\bibitem{GaSt09}
D.T.~Gay and A.I.~Stipsicz.
\newblock Symplectic surgeries and normal surface singularities.
\newblock {\em Algebr. Geom. Topol.}, 9(4):2203-2223, 2009.

\bibitem{GoLi14}
M.~Golla and P.~Lisca.
\newblock On Stein fillings of contact torus bundles.
\newblock  Bull. Lond. Math. Soc. 48 (2016) 19-37.

\bibitem{Gom95}
R.E.~Gompf.
\newblock A new construction of symplectic manifolds.
\newblock {\em Ann. of Math.}, 143(3):527-595, 1995.

\bibitem{GrHaKe11}
M.~Gross, P.~Hacking and S.~Keel.
\newblock Mirror symmetry for log Calabi-Yau surfaces I.
\newblock arXiv:1106.4977, 2011.

\bibitem{GrHaKe12}
M.~Gross, P.~Hacking and S.~Keel.
\newblock Moduli of surfaces with an anti-canonical Cy.
\newblock arXiv:1211.6367, 2012.





\bibitem{LaMc96-2}
F.~Lalonde and D.~McDuff.
\newblock J-curves and the classification of ruled symplectic 4-manifolds.
\newblock {\em Contact and symplectic geometry Publ. Newton Inst. 8}, (Cambridge, 1994), 3-42, Cambridge Univ. Press, Cambridge, 1996




\bibitem{LiMa14}
T-J. Li and C.Y. Mak.
\newblock Symplectic Divisorial Capping in Dimension 4.
\newblock arXiv:1407.0564.

\bibitem{LiMa16-deformation}
T-J. Li and C.Y. Mak.
\newblock Symplectic log Calabi-Yau surface--deformation class. 
\newblock Adv. Theor. Math. Phys. 20 (2016), no. 2, 351-379.


\bibitem{LiMa16}
T-J. Li and C.Y. Mak.
\newblock The Kodaira dimension of contact 3-manifolds and geography of symplectic fillings
\newblock arXiv: 1610.06870.

\bibitem{LiMa17-contact}
T-J. Li and C.Y. Mak.
\newblock Symplectic log Calabi-Yau surface--contact aspects. 
\newblock In preparation.

\bibitem{LiMaYa14}
T-J. Li, C.Y. Mak and K. Yasui.
\newblock  Calabi-Yau caps,   uniruled caps and symplectic fillings.
\newblock {\em Proc. Lond. Math. Soc.}, (3) 114, (2017), no. 1,  159-187.      




\bibitem{LiZh11}
T-J. Li and W. Zhang.
\newblock Additive and relative Kodaira dimension.
\newblock {\em Geometry and Analysis, No.2} 18:103-135, Adv. Lect. Math.(ALM), Int. Press., 2011.



\bibitem{Liu96}
A. Liu.
\newblock Some new applications of general wall crossing formula, Gompf's conjecture and its applications. 
\newblock {\em Math. Res. Lett.} 3(5): 569-585, 1996. 

\bibitem{Lo81}
E.~Looijenga.
\newblock Rational surfaces with an anti-canonical divisor.
\newblock {\em Ann. of Math.} 114(2): 267-322, 1981.




\bibitem{McOp13}
D.~McDuff and E.~Opshtein.
\newblock Nongeneric J-holomorphic curves and singular inflation.
\newblock {\em Algebr. Geom. Topol.} 15: 231-286, 2015.

\bibitem{McL12}
M.~McLean.
\newblock The growth rate of symplectic homology and affine varieties.
\newblock {\em Geometric And Functional Analysis.}, 22(2): 369-442, 2012.

\bibitem{Ne81}
W.D.~Neumann.
\newblock A calculus for plumbing applied to the topology of complex surface singularities and degenerating complex curves.
\newblock {\em TAMS},   268 (2):299-344, 1981.

\bibitem{OhOn96}
H.~Ohta and K.~Ono.
\newblock   Notes on symplectic 4-manifolds with $b^+=1$,  II.
\newblock {\em Int. J. Math. } 7: 755-770, 1996. 

\bibitem{OhOn05}
H.~Ohta and K.~Ono.
\newblock Simple singularities and symplectic fillings.
\newblock {\em J. Diff. Geom.}, 69(1):1-42, 2005.

\bibitem{OhOn05-cuspidal}
H.~Ohta and K.~Ono.
\newblock Symplectic 4-manifolds containing singular rational curves with $(2, 3)-$cusp
\newblock {\em Seminares \& Congres} 10 (2005), 233-241.

\bibitem{OhOn03}
H.~Ohta and K.~Ono.
\newblock Symplectic fillings of the link of simple elliptic singularities. 
\newblock {\em J. Reine Angew. Math.} 565 (2003), 183-205

\bibitem{OhOn09}  
H.~Ohta and K.~Ono.
\newblock An inequality for symplectic fillings of the link of a hypersurface K3 singularity. Algebraic topology: old and new, 93-100, \newblock Banach Center Publ., 85, Polish Acad. Sci. Inst. Math., Warsaw, 2009

\bibitem{Pa13}
J.~Pascaleff.
\newblock On the symplectic cohomology of log Calabi-Yau surfaces.
\newblock arXiv:1304.5298, 2013.

\bibitem{Pi08}
M.~Pinsonnault.
\newblock Maximal compact tori in the Hamiltonian group of 4-dimensional symplectic manifolds.
\newblock {\em J. Mod. Dyn.}, 2(3): 431-455, 2008.

\bibitem{Sa13}
D.~Salamon.
\newblock Uniqueness of symplectic structures.
\newblock {\em Acta Math. Vietnam.}, 38(1): 123-144, 2013.





\bibitem{Ta95}
C.H.~Taubes.
\newblock The Seiberg-Witten and Gromov invariants.
\newblock {\em Math. Res. Lett.} 2(2): 221-238, 1995.

\bibitem{Us09}
M.~Usher.
\newblock Kodaira dimension and symplectic sums.
\newblock {\em Comment. Math. Helv.} 1(84): 57-85, 2009.

\end{thebibliography}
\bibliographystyle{plain}

\end{document}